\documentclass[12pt]{amsart}
\usepackage{amssymb}
\usepackage{amsmath, amscd}
\usepackage{amsthm}
\usepackage[table,xcdraw]{xcolor}
\usepackage{ulem}
\usepackage{tikz}
\usepackage{circuitikz}
\usepackage[colorlinks=true, linkcolor=blue,urlcolor=blue]{hyperref}

\newtheorem{theorem}{Theorem}[section]

\newtheorem{proposition}[theorem]{Proposition}
\newtheorem{lemma}[theorem]{Lemma}

\newtheorem{corollary}[theorem]{Corollary}
\theoremstyle{definition}

\newtheorem{example}[theorem]{Example}

\newtheorem{remark}[theorem]{Remark}

\begin{document}

\author[P. Danchev]{Peter Danchev}
\address{Institute of Mathematics and Informatics, Bulgarian Academy of Sciences, 1113 Sofia, Bulgaria}
\email{danchev@math.bas.bg; pvdanchev@yahoo.com}
\author[O. Hasanzadeh]{Omid Hasanzadeh}
\address{Department of Mathematics, Tarbiat Modares University, 14115-111 Tehran Jalal AleAhmad Nasr, Iran}
\email{hasanzadeomiid@gmail.com}
\author[A. Moussavi]{Ahmad Moussavi}
\address{Department of Mathematics, Tarbiat Modares University, 14115-111 Tehran Jalal AleAhmad Nasr, Iran}
\email{moussavi.a@modares.ac.ir; moussavi.a@gmail.com}

\title[Rings whose clean elements are uniquely strongly clean]{\small Rings whose clean elements are \\ uniquely strongly clean}

\begin{abstract} We define the class of {\it CUSC} rings, that are those rings whose clean elements are uniquely strongly clean. These rings are a common generalization of the so-called {\it USC} rings, introduced by Chen-Wang-Zhou in J. Pure \& Applied Algebra (2009), which are rings whose elements are uniquely strongly clean. These rings also generalize the so-called {\it CUC} rings, defined by Calugareanu-Zhou in Mediterranean J. Math. (2023), which are rings whose clean elements are uniquely clean. We establish that a ring is USC if, and only if, it is simultaneously CUSC and potent. Some other interesting relationships with CUC rings are obtained as well.
\end{abstract}

\maketitle




\section{Introduction and Basic Notion}

Everywhere in the current paper, let $R$ be an associative but {\it not} necessarily commutative ring having identity element, usually stated as $1$. Standardly, for such a ring $R$, the letters $U(R)$, $\rm{Nil}(R)$ and ${\rm Id}(R)$ are designed for the set of invertible elements (also termed as the unit group of $R$), the set of nilpotent elements and the set of idempotent elements in $R$, respectively. Likewise, $J(R)$ denotes the Jacobson radical of $R$. For all other unexplained explicitly notions and notations, we refer to the classical source \cite{6}.

In order to present our achievements here, we now need the necessary background material as follows: Mimicking \cite{5}, an element $a$ from a ring $R$ is called {\it clean} if there exists $e\in {\rm Id}(R)$ such that $a-e\in U(R)$. Then, $R$ is said to be {\it clean} if each element of $R$ is clean. In addition, $a$ is called {\it strongly clean} provided $ae=ea$ and, if each element of $R$ are strongly clean, then $R$ is said to {\it strongly clean} too.
On the other hand, imitating \cite{3}, $a\in R$ is called {\it uniquely clean} if there exists a unique $e \in {\rm Id}(R)$ such that $a-e \in U(R)$. In particular, a ring $R$ is said to be {\it uniquely clean} (or just {\it UC} for short) if every element in $R$ is uniquely clean.

Generalizing these notions, in \cite{2} was defined an element $a \in R$ to be {\it uniquely strongly clean} if there exists a unique $e \in {\rm Id}(R)$ such that $a-e \in U(R)$ and $ae=ea$. In particular, a ring $R$ is {\it uniquely strongly clean} (or just {\it USC} for short) if each element in $R$ is uniquely strongly clean.

In a similar vein, expanding the first part of the above concepts, in \cite{1} a ring is called {\it CUC} if any clean element is uniquely clean, and a ring is called {\it UUC} if any unit is uniquely clean.

Our work targets to extend these two last definitions by defining and exploring the following key instruments: A ring is called {\it CUSC} ring if every clean element is uniquely strongly clean, and a ring is called {\it UUSC} ring if every unit is uniquely strongly clean.\\

The relations between all of these notions are showed by the next diagrams:

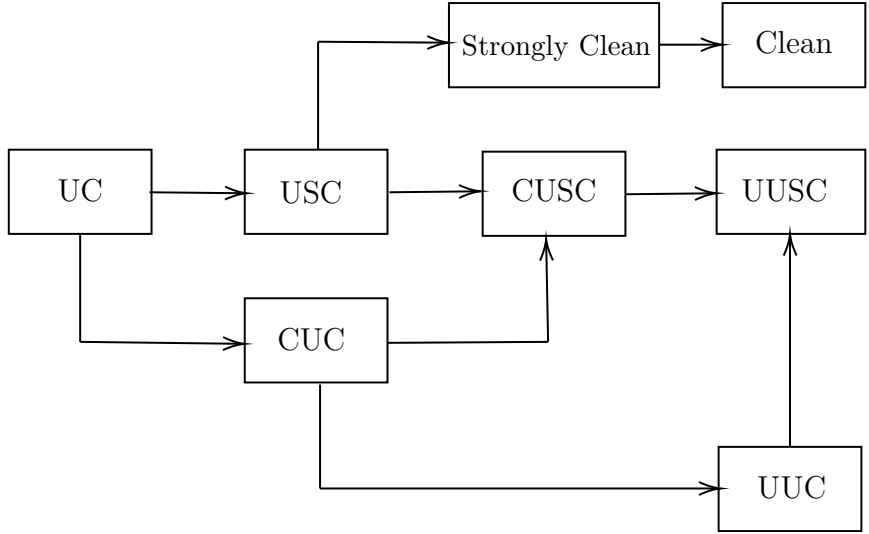
\begin{center}

\tikzset{every picture/.style={line width=0.75pt}} 

\begin{tikzpicture}[x=0.75pt,y=0.75pt,yscale=-1,xscale=1]

\draw   (76,152.5) -- (148,152.5) -- (148,195) -- (76,195) -- cycle ;
\draw   (195,152.5) -- (267,152.5) -- (267,195) -- (195,195) -- cycle ;
\draw   (315,153.5) -- (387,153.5) -- (387,196) -- (315,196) -- cycle ;
\draw   (433,152.5) -- (508,152.5) -- (508,195) -- (433,195) -- cycle ;
\draw   (195,227.5) -- (267,227.5) -- (267,270) -- (195,270) -- cycle ;
\draw   (434,302.5) -- (506,302.5) -- (506,345) -- (434,345) -- cycle ;
\draw   (436,78.5) -- (508,78.5) -- (508,121) -- (436,121) -- cycle ;
\draw   (298,78.5) -- (404,78.5) -- (404,121) -- (298,121) -- cycle ;
\draw    (147,174) -- (194,174.48) ;
\draw [shift={(196,174.5)}, rotate = 180.58] [color={rgb, 255:red, 0; green, 0; blue, 0 }  ][line width=0.75]    (10.93,-3.29) .. controls (6.95,-1.4) and (3.31,-0.3) .. (0,0) .. controls (3.31,0.3) and (6.95,1.4) .. (10.93,3.29)   ;
\draw    (268,174) -- (312,173.52) ;
\draw [shift={(314,173.5)}, rotate = 179.38] [color={rgb, 255:red, 0; green, 0; blue, 0 }  ][line width=0.75]    (10.93,-3.29) .. controls (6.95,-1.4) and (3.31,-0.3) .. (0,0) .. controls (3.31,0.3) and (6.95,1.4) .. (10.93,3.29)   ;
\draw    (404,99.5) -- (434,99.5) ;
\draw [shift={(436,99.5)}, rotate = 180] [color={rgb, 255:red, 0; green, 0; blue, 0 }  ][line width=0.75]    (10.93,-3.29) .. controls (6.95,-1.4) and (3.31,-0.3) .. (0,0) .. controls (3.31,0.3) and (6.95,1.4) .. (10.93,3.29)   ;
\draw    (112,195) -- (112,249.5) ;
\draw    (112,249.5) -- (193,250.48) ;
\draw [shift={(195,250.5)}, rotate = 180.69] [color={rgb, 255:red, 0; green, 0; blue, 0 }  ][line width=0.75]    (10.93,-3.29) .. controls (6.95,-1.4) and (3.31,-0.3) .. (0,0) .. controls (3.31,0.3) and (6.95,1.4) .. (10.93,3.29)   ;
\draw    (233,271) -- (233,298.5) -- (233,323.5) ;
\draw    (233,323.5) -- (433,323.5) ;
\draw [shift={(435,323.5)}, rotate = 180] [color={rgb, 255:red, 0; green, 0; blue, 0 }  ][line width=0.75]    (10.93,-3.29) .. controls (6.95,-1.4) and (3.31,-0.3) .. (0,0) .. controls (3.31,0.3) and (6.95,1.4) .. (10.93,3.29)   ;
\draw    (232,98) -- (232,152.5) ;
\draw    (232,98) -- (296,98.48) ;
\draw [shift={(298,98.5)}, rotate = 180.43] [color={rgb, 255:red, 0; green, 0; blue, 0 }  ][line width=0.75]    (10.93,-3.29) .. controls (6.95,-1.4) and (3.31,-0.3) .. (0,0) .. controls (3.31,0.3) and (6.95,1.4) .. (10.93,3.29)   ;
\draw    (267,250) -- (348,249.5) ;
\draw    (348,249.5) -- (347.04,199.5) ;
\draw [shift={(347,197.5)}, rotate = 88.9] [color={rgb, 255:red, 0; green, 0; blue, 0 }  ][line width=0.75]    (10.93,-3.29) .. controls (6.95,-1.4) and (3.31,-0.3) .. (0,0) .. controls (3.31,0.3) and (6.95,1.4) .. (10.93,3.29)   ;
\draw    (470,195) -- (470,302.5) ;*********
\draw [shift={(470,195)}, rotate = 89.55] [color={rgb, 255:red, 0; green, 0; blue, 0 }  ][line width=0.75]    (10.93,-3.29) .. controls (6.95,-1.4) and (3.31,-0.3) .. (0,0) .. controls (3.31,0.3) and (6.95,1.4) .. (10.93,3.29)   ;
\draw    (387,175) -- (431,174.52) ;
\draw [shift={(433,174.5)}, rotate = 179.38] [color={rgb, 255:red, 0; green, 0; blue, 0 }  ][line width=0.75]    (10.93,-3.29) .. controls (6.95,-1.4) and (3.31,-0.3) .. (0,0) .. controls (3.31,0.3) and (6.95,1.4) .. (10.93,3.29)   ;

\draw (99,166) node [anchor=north west][inner sep=0.75pt]   [align=left] {UC};
\draw (211,167) node [anchor=north west][inner sep=0.75pt]   [align=left] {USC};
\draw (328,166) node [anchor=north west][inner sep=0.75pt]   [align=left] {CUSC};
\draw (451,91) node [anchor=north west][inner sep=0.75pt]   [align=left] {Clean};
\draw (444,166) node [anchor=north west][inner sep=0.75pt]   [align=left] {UUSC};
\draw (452,316) node [anchor=north west][inner sep=0.75pt]   [align=left] {UUC};
\draw (210,241) node [anchor=north west][inner sep=0.75pt]   [align=left] {CUC};
\draw (303,93) node [anchor=north west][inner sep=0.75pt]  [font=\small] [align=left] {Strongly Clean};

\end{tikzpicture}

\end{center}

Our plan to do that is structured thus: In the next section, we give some critical examples and properties of both CUSC and UUSC rings and their extensions (see, for instance, Propositions~\ref{prop2.2} and \ref{formal}). The subsequent third section is devoted to the exhibition of some other remarkable characterization behaviors of these two classes of rings under some standard extending procedures like matrix rings and group rings (see Theorems~\ref{theo3.1}, \ref{theo3.10}, \ref{theo3.4}, \ref{theo3.11}, \ref{prop4.3}). We also pose some challenging properties which, hopefully, will stimulate a further research on the subject.

\section{Examples and Major Properties}

We start our consideration with the following.

\begin{example}\label{exam1.4} The following are fulfilled:
\hskip 1cm
\begin{enumerate}
\item
$\mathbb{Z}_2[x]$ is CUSC, but is {\it not} USC.

\item
$\mathbb{Z}[x]$ is CUC, but is {\it not} UC.

\item
$R=\begin{pmatrix}
\mathbb{Z}_2 & \mathbb{Z}_2 \\
0 & \mathbb{Z}_2
\end{pmatrix}$ is UUC, but is {\it not} CUC.

\item
${\rm T}_2(\mathbb{Z}_2)$ is CUSC, but is {\it not} CUC.
\end{enumerate}
\end{example}

\begin{proof}
\hskip 1cm
\begin{enumerate}
\item
By \cite[Example 2.2 (1)]{1}, $\mathbb{Z}_2[x]$ is CUC, and hence so is CUSC. However, it is not USC; otherwise, if $\mathbb{Z}_2[x]$ is USC, then it has to be clean, but this contradicts \cite[Example 2]{5}.

\item
$\mathbb{Z}[x]$ is CUC by \cite[Example 2.2(1)]{1}, but is not UC; otherwise, if $\mathbb{Z}[x]$ is UC, then it should be clean, which is a contradiction with \cite[Example 2]{5}.

\item
$R$ is UUC by \cite[Corollary 2.13(5)]{1}, but is not CUC, because CUC rings are always abelian in view of \cite[Proposition 2.1]{1}. If $R$ is CUC, then
$\begin{pmatrix}
1 & 0\\
0 & 0
\end{pmatrix}$
$\in {\rm Id} (R)$ is central. But
$\begin{pmatrix}
0 & 1\\
0 & 0
\end{pmatrix}
\begin{pmatrix}
1 & 0\\
0 & 0
\end{pmatrix}$
$ = 0$
 and
$\begin{pmatrix}
1 & 0\\
0 & 0
\end{pmatrix}
\begin{pmatrix}
0 & 1\\
0 & 0
\end{pmatrix}
=
\begin{pmatrix}
0 & 1\\
0 & 0
\end{pmatrix},$
a contradiction.
\item
${\rm T}_2(\mathbb{Z}_2)$ is CUSC in virtue of \cite[Theorem 10]{2}, but is not CUC, because CUC rings are always abelian.
\end{enumerate}
\end{proof}

We now show the validity of a series of technical claims.

\begin{proposition}\label{prop2.1} 
Let $R$ be an abelian ring. Then, the following are equivalent:

\begin{enumerate}
\item
$(U(R)+U(R)) \cap {\rm Id}(R)=\{0\}$.
\item
$R$ is UUSC.
\item
$R$ is UUC.
\item
$R$ is CUC.
\item
$R$ is CUSC.
\end{enumerate}
\end{proposition}

\begin{proof}
(i) $\Rightarrow$ (ii): Let $u = e + v$ where $ev = ve$, $u, v  \in U (R)$ and $e \in {\rm Id}(R)$. Then we have $e = u + (-v)$ and by (i), $e = 0$.\\
(ii) $\Rightarrow$ (iii): Let $u = e + v$ where $u,v \in U(R)$ and $e \in {\rm Id} (R)$. By hypothesis we have $ev = ve$ and hence by (ii) we get $e = 0$.\\
(iii) $\Rightarrow$ (iv): Follows from \cite[Proposition 2.1]{1}.\\
(iv) $\Rightarrow$ (v): It is clear.\\
(v) $\Rightarrow$ (i): Let $e = u_1 + u_2$ where $u_1, u_2  \in U (R)$ and $e \in {\rm Id} (R)$, whence $u_1 = e + (-u_2) = 0 + u_1$, so we have two strongly clean decompositions for $u_1$. Then $e = 0$ by (v).
\end{proof}

An element in a ring is called {\it 2-good} if it is the sum of two units.

\begin{example}
\hskip 1cm
\begin{enumerate}
\item
Let ${\rm Id} (R) = \{0, 1 \}$. Then, $R$ is CUSC (resp., UUSC) if, and only if, $1$ is {\it not} 2-good.\\

\item
Let $R$ be a commutative ring. If $R$ is USC (or, resp., UC), then ${\rm T}_n (R)$ is USC for any $n \geq 2$ using \cite[Theorem]{2}. So, ${\rm T}_n (R)$ is CUSC for any $n \geq 2$. But, ${\rm T}_n (R)$ is not CUC for any ring $R$ and any $n \geq 2$, because CUC rings are always abelian.
\end{enumerate}
\end{example}

A subring of a USC ring need {\it not} be USC; for example, $\mathbb{Z}_2[[x]]$ is USC, but $\mathbb{Z}_2[x]$ is not USC. Nevertheless, we obtain the following.

\begin{proposition}\label{prop 2.4}
Let $S$ be a subring of a ring $R$. If $R$ is CUSC (resp., UUSC), then so is $S$.
\end{proposition}

\begin{proof}
In view of Proposition \ref{prop2.1}, it suffices to show the UUSC case. To that aim, write $u=v+e$, where $u,v\in U(S)$ and $e\in Id(S)$ with $ev=ve$. Then, $$u-(1_{R}-1_{S})=[v-(1_{R}-1_{S})]+e,$$ where both $u-(1_{R}-1_{S}), v-(1_{R}-1_{S})\in U(R)$ and $e\in {\rm Id}(S)\subseteq {\rm Id}(R)$ with $$e[v-(1_{R}-1_{S})]=ev-e(1_{R}-1_{S})=ve -(1_{R}-1_{S})e =[v-(1_{R}-1_{S})]e.$$ Since $R$ is UUSC, it follows at once that $e=0$. So, $S$ is UUSC, as promised.
\end{proof}

\begin{proposition}\label{prop 2.5}
Let $R = \prod_{i \in I} R_i$ be a direct product of rings $\{R_i\}$ for all $i \in I$. Then, $R$ is CUSC (resp., UUSC) if, and only if, so are $R_i$ for all indices $i$.
\end{proposition}

\begin{proof}
The necessity follows quickly from Proposition \ref{prop 2.4}.

For the sufficiency, let $(u_{i})=(v_{i})+ (e_{i})$, where $(e_{i}) \in {\rm Id}(R)$ and $(u_{i}), (v_{i}) \in U(R)$ with $(v_{i})(e_{i})=(e_{i})(v_{i})$. Then, for each index $i$, we write $u_{i}=v_{i}+e_{i}$, where $e_{i}\in {\rm Id}(R_{i})$ and $u_{i},v_{i}\in U(R_{i})$ with $v_{i}e_{i}=e_{i}v_{i}$. Since $R_{i} $ is UUSC, it follows at once that $e_{i}=0$. So, $(e_{i})=0$. Hence, $R$ is UUSC, as expected.
\end{proof}

Two immediate consequences are the following.

\begin{corollary}
Any subdirect product of CUSC (resp., UUSC) rings is CUSC (resp., UUSC).
\end{corollary}

\begin{corollary}
Let $R$ be a ring and $e$ a central idempotent of $R$. Then, the following are equivalent:
\begin{enumerate}
\item
$R$ is CUSC (resp., UUSC)

\item
$eR$ and $(1-e) R$ are CUSC (resp., UUSC).
\end{enumerate}
\end{corollary}

\begin{proof}
It follows directly from Propositions \ref{prop 2.4} and \ref{prop 2.5}.
\end{proof}

\begin{lemma}
Let $I \subseteq J(R)$ be an ideal of $R$. The following hold:
\begin{enumerate}
\item
If $R/I$ is UUSC, then $R$ is UUSC. The converse holds, provided $R$ is an abelian ring and idempotents lift modulo $I$.

\item
If $R/I$ is CUSC and $R$ is abelian, then $R$ is CUSC. The converse holds, provided idempotents lift modulo $I$.
\end{enumerate}
\end{lemma}

\begin{proof}
\hskip 0cm
\begin{enumerate}
\item
Write $u = e + v$, where $u, v \in U (R)$ and $e \in {\rm Id} (R)$ with $ev = ve$, so it is immediate that $\bar{u} = \bar{e} + \bar{v}$, where $\bar{u}, \bar{v} \in U (\bar{R} =\dfrac{R}{I})$ and $\bar{e} \in {\rm Id} (\bar{R})$ with $\bar{e} \bar{v} = \bar{v} \bar{e}$; here $\bar{R} := R/I$. Consequently, we have $\bar{e} = \bar{0}$, so that $e \in I \subseteq J(R)$ and hence $e = 0$.

Conversely, let $\bar{u} = \bar{e} + \bar{v}$, where $\bar{u}, \bar{v} \in U (\bar{R} =\dfrac{R}{I})$ and $\bar{e} \in {\rm Id} (\bar{R})$. It suffices  to show that $\bar{e} = \bar{0}$. By hypothesis, we can assume $e \in {\rm Id}(R)$ and $u, v \in U (R)$. Then, we have $u - (e + v) \in I \subseteq J(R)$, so $u = e + v + j = e + (v + j)$ for $j \in J(R)$. As $R$ is abelian, we obtain $e(v + j) = (v + j) e$, where $v + j \in U (R)$. Then, $e = 0$ and hence $\bar{e} = \bar{0}$.\\

\item
The proof is similar to (i).
\end{enumerate}
\end{proof}

A ring $R$ is called {\it local} if $R/J(R)$ is a division ring.

\begin{proposition}\label{prop2.2}
Let $R$ be a ring. Then, the following are equivalent:

\begin{enumerate}
\item
$R$ is CUSC and local.
\item
$\dfrac{R}{J(R)} \cong \mathbb{Z}_2$.
\item
$R$ is UC and local.
\item
$R$ is UC and ${\rm Id}(R)=\{0,1\}$.
\item
$R$ is USC and local.
\item
$R$ is USC and ${\rm Id} (R) = \{0, 1\}$.
\item
$R$ is CUC and local.
\end{enumerate}
\end{proposition}

\begin{proof}
(i) $\Rightarrow$ (iii): Since $R$ is local, it must be that ${\rm Id}(R)=\{0, 1\}$. Thus, $R$ is abelian and also is CUSC, so $R$ is both CUC and clean (because every local ring is known to be clean). Therefore, $R$ is uniquely clean.\\
(iii) $\Rightarrow$ (i): It follows directly from Proposition \ref{prop2.1} .\\
(ii) $\Leftrightarrow$ (iii) $\Leftrightarrow$ (iv): The equivalences are proved in \cite[Theorem 15]{3}.\\
(iii) $\Leftrightarrow$ (v): It is clear owing to \cite[Example 4]{2}.\\
(iv) $\Leftrightarrow$ (vi): It is clear.\\
(i) $\Leftrightarrow$ (vii): It is clear by Proposition \ref{prop2.1}.
\end{proof}

For a subring $C$ of a ring $D$, the set
$$R[D,C]:=\{ (d_{1},\ldots ,d_{n},c,c,\ldots)| d_{i}\in D,c\in C,n\geq 1\}$$
with the addition and the multiplication defined component-wise is called the {\it tail ring extension} and is denoted by $R[D,C]$.

\begin{proposition}
The ring $R[D,C]$ is CUSC (resp., UUSC) if, and only if, so is $D$.
\end{proposition}

\begin{proof}
Straightforward.
\end{proof}

If $R$ is a ring and $\alpha :R\rightarrow R$ is a ring homomorphism, let $R[[x,\alpha ]]$ denote the ring of {\it skew formal power series} over $R$; that is, all formal power series in $x$ with coefficients from $R$ with multiplication defined by $xr=\alpha (r)x$ for all $r\in R$.

We now offer the following:

\begin{proposition}\label{formal}
$R[[x, \alpha]]$ is UUSC if, and only if, $R$ is UUSC.
\end{proposition}

\begin{proof}
Assuming that $R[[x, \alpha]]$ is a UUSC ring, as we know that $$U(R[[x, \alpha]]) = \{ \sum_{i=0}^{\infty}a_ix^i \mid a_0 \in U(R) \},$$ we can readily conclude that $R$ is a UUSC ring.

Reciprocally, let us now assume that $R$ is a UUSC ring and $$f = \sum_{i=0}^{\infty} f_ix^i \in U(R[[x, \alpha]]).$$ Apparently, $f_0 \in U(R)$. Now, if we have $f = e + g$, where $e \in {\rm Id}(R[[x, \alpha]])$, $g \in U(R[[x, \alpha]])$ and $eg = ge$, then $g_0$ is a unit and $e_0$ is an idempotent in $R$ such that $f_0 = e_0 + g_0$ and $e_0 g_0= g_0e_0$. Since $R$ is a UUSC ring, we infer $e_0 = 0$. From $e^2 = e$ and $e_0 = 0$, we can conclude with the aid of a simple calculation that $e_i = 0$, which implies that $e = 0$, as required.
\end{proof}

As an interesting consequence, we deduce:

\begin{corollary}
Let $R$ be a ring and $\alpha : R \rightarrow R$ is a ring homomorphism. Then, the following hold:
\begin{enumerate}
\item
For $n \geq 1, \frac{R[[x, \alpha]]}{\langle x^n \rangle}$ is UUSC if, and only if, $R$ is UUSC.
\item
For $n \geq 1, \frac{R[x, \alpha]}{\langle x^n \rangle}$ is UUSC if, and only if, $R$ is UUSC.
\end{enumerate}
\end{corollary}

We are now ready to attack the following.

\begin{proposition}\label{prop 2.13}
Suppose that $R= S+I$, where $S$ is a subring of $R$ with $1_{R}\in S$, and $I$ is an ideal of $R$ such that $S\cap I=\{0\}$. If $S$ is UUSC and $(U(R) + U(R))\cap {\rm Id}(I) = \{0\}$, then $R$ is UUSC.
\end{proposition}

\begin{proof}
Write $u=e+v$, where $ev=ve$, $u,v\in U(R)$ and $e\in Id(R)$. Write $u=u_{s}+u_{i}$, $e=e_{s}+e_{i}$, $v=v_{s}+v_{i}$, where $u_{s},e_{s},v_{s}\in S$ and $u_{i},e_{i},v_{i}\in I$. So, employing \cite[Proposition 2.12]{1}, we deduce $u_{s} = e_{s} + v_{s}$, where $u_{s},v_{s}\in U(S)$, $e_{s}\in {\rm Id}(S)$. Also, one may write $$ev= (e_{s}+e_{i}) (v_{s}+v_{i}) = e_{s} v_{s} + e_{s}v_{i}+e_{i}v_{s}+e_{i}v_{i}=$$
$$= v_{s}e_{s} + v_{s}e_{i}+ v_{i}e_{s} + v_{i}e_{i}=(v_{s}+ v_{i}) (e_{s}+e_{i})=ve.$$ If, however, $e_{s} v_{s}=v_{s}e_{i}$ or $v_{i}e_{s}=v_{i}e_{i}$, then one sees that $e_{s}v_{s}=0$, because $S\cap I=\{0\}$. Therefore, $e_{s}=0$, and thus $e=e_{i}\in I$. From the equality $(U(R)+U(R))\cap Id(I)=\{0\}$, it follows automatically that $e=0$. But, if $e_{s}v_{s}=v_{s}e_{s}$, we derive $e_{s}=0$, because $S$ is UUSC. So, by what we have established above, we receive $e=0$. Processing analogously, we will have a similar result for double, triple and quadruple states. Therefore, in all cases, we obtain $e_{s}=0$ and hence $e=0$, so that $R$ is UUSC, as claimed.
\end{proof}

Let $A$ be a ring and $V$ a bi-module over $A$. The trivial extension of $A$ and $V$ is ${\rm T}(A, V) = \{(a, x) | a \in A, x \in V\}$ with the addition defined component-wise and the multiplication defined by $(a, x)(b, y) = (ab, ay + xb)$.

\begin{corollary}\label{cor 2.14}
Let $A, B$ be rings, $V$ a bi-module over $A, M$ an $(A, B)$-bi-module and $n \geq 1$

\begin{enumerate}
\item
$A[[x]]$ is UUSC if, and only if, $A$ is UUSC.

\item
$\dfrac{A[[x]]}{\langle x^n\rangle}$ is UUSC if, and only if, $A$ is UUSC.

\item
$\frac{A[x]}{\langle x^n\rangle}$ is UUSC if, and only if, $A$ is UUSC.

\item
${\rm T}(A, V)$ is UUSC if, and only if, $A$ is UUSC.

\item
The formal triangular matrix ring
$\begin{pmatrix}
A & M\\
0 & B
\end{pmatrix}$
is UUSC if, and only if, both $A, B$ are UUSC.

\item
The upper triangular matrix ring ${\rm T}_n(A)$ is UUSC if, and only if, $A$ is UUSC.
\end{enumerate}
\end{corollary}

\begin{proof}
\begin{enumerate}
\item
Let $R = A[[x]], S = A$ and $I = xR$. Since ${\rm Id} (I) = \{0\}$, point (i) follows from Propositions \ref{prop 2.13} and \ref{prop 2.4}.

\item
Let $R = \frac{A[[x]]}{\langle x^n\rangle}$, $S = A$ and $I = xR$. Since ${\rm Id}(I) = \{0\}$, point (ii) follows from Propositions \ref{prop 2.13} and \ref{prop 2.4}.

\item
As $\frac{A[[x]]}{\langle x^n\rangle} \cong \frac{A[x]}{\langle x^n \rangle}$, point (iii) follows automatically from (ii).

\item
For the sufficiency, let $R = {\rm T}(A, V)$, $S = {\rm T}(A, 0)$ and $I = {\rm T}(0, V)$. Since ${\rm Id} (I) = \{0\}$, $R$ is UUSC utilizing Proposition \ref{prop 2.13}.

For the necessity, let $R$ be UUSC and $u = e + v$, where $u, v \in U (A)$ and $e \in {\rm Id} (A)$ with $ev = ve$. Then, we have $(u, 0) = (e, 0) + (v, 0)$, where $(u, 0)(v, 0) \in U (R)$ and $(e, 0) \in {\rm Id} (R)$ with $(e, 0)(v, 0) = (v, 0)(e, 0)$. So, $(e, 0) = (0,0)$ and hence $e = 0$. Thus, $A$ is UUSC.

\item
For the sufficiency, let $R =
\begin{pmatrix}
A & M\\
0 & B
\end{pmatrix}$,
$S=
\begin{pmatrix}
A & 0\\
0 & B
\end{pmatrix}$
and $I  =
\begin{pmatrix}
0 & M\\
0 & 0
\end{pmatrix}$.
Since ${\rm Id} (I) = \{0\}$, $R$ is UUSC consulting with Proposition \ref{prop 2.13}.

For the necessity, let $R$ be UUSC. Write $u = e + v$ and $u^\prime = e^\prime + v^\prime$, where $u, v \in U (A), u^\prime, v^\prime \in U (B), e \in {\rm Id} (A)$ and $e^\prime \in {\rm Id} (B)$ with $ev = ve$ and $e^\prime v^\prime = v^\prime e^\prime$. Therefore, one infers that
$$\begin{pmatrix}
u & 0\\
0 & u^\prime
\end{pmatrix}
=
\begin{pmatrix}
e & 0\\
0 & e^\prime
\end{pmatrix}
+
\begin{pmatrix}
v & 0\\
0 & v^\prime
\end{pmatrix},$$
\noindent where
$\begin{pmatrix}
u & 0\\
0 & u^\prime
\end{pmatrix},
\begin{pmatrix}
v & 0\\
0 & v^\prime
\end{pmatrix}
\in U (R)$ and
$\begin{pmatrix}
e & 0\\
0 & e^\prime
\end{pmatrix}
\in {\rm Id}(R)$ with
$$\begin{pmatrix}
e & 0\\
0 & e^\prime
\end{pmatrix}
\begin{pmatrix}
v & 0\\
0 & v^\prime
\end{pmatrix}
=
\begin{pmatrix}
v & 0\\
0 & v^\prime
\end{pmatrix}
\begin{pmatrix}
e & 0\\
0 & e^\prime
\end{pmatrix}.$$
\noindent So, by hypothesis,
$$\begin{pmatrix}
e & 0\\
0 & e^\prime
\end{pmatrix}
=
\begin{pmatrix}
0 & 0\\
0 & 0
\end{pmatrix}.$$
Consequently, $e = e^\prime = 0$ whence $A, B$ are both UUSC.

\item
It follows from (v).
\end{enumerate}
\end{proof}

As a consequence, we derive:

\begin{corollary}
Let $R$ be a ring and let $T$ be a subring of $R[[x]]$ with $R \subseteq T \subseteq R[[x]]$. Then, $R$ is UUSC if, and only if, so is $T$.
\end{corollary}

\begin{proof}
It follows from a combination of Proposition \ref{prop 2.4} and Corollary \ref{cor 2.14}.
\end{proof}

A Morita context is a $4-tuple
\begin{pmatrix}
A & M\\
N & B
\end{pmatrix}$,
where $A$ and $B$ are rings, ${}_{A}M_{B}$ and ${}_{B}N_{A}$ are bi-modules and there exist two context products $M \times N \rightarrow A$ and $N \times M \rightarrow B$, written multiplicatively as $(x, y) = xy$ and $(y, x) = yx$, such that
$\begin{pmatrix}
A & M\\
N & B
\end{pmatrix}$
is an associative ring with the obvious matrix operations. A Morita context
$\begin{pmatrix}
A & M\\
N & B
\end{pmatrix}$
is called trivial, provided the context products are trivially, i.e., $MN = 0$ and $NM = 0$. We know
$$\begin{pmatrix}
A & M\\
N & B
\end{pmatrix}
\cong
{\rm T}(A \times B, M \oplus N),$$
where
$\begin{pmatrix}
A & M\\
N & B
\end{pmatrix}$
is a trivial Morita context.

We, thereby, extract the following.

\begin{corollary}
The trivial Morita context
$\begin{pmatrix}
A & M\\
N & B
\end{pmatrix}$
is UUSC if, and only if, $A, B$ are both UUSC.
\end{corollary}

\begin{proof}
It is easy to see that the next relations are valid:
$$\begin{pmatrix}
A & M\\
N & B
\end{pmatrix}
\cong {\rm T}(A \times B, M \oplus N) \cong
\begin{pmatrix}
A \times B & M \oplus N\\
0 & A \times B
\end{pmatrix}.$$
Thus, the rest of the proof follows applying Corollary \ref{cor 2.14} and Proposition \ref{prop 2.5}.
\end{proof}

\begin{proposition}
If ${\rm T}(A, V)$ is CUSC, then $A$ is CUSC. The converse holds, provided that $ex = xe$ for all $x \in V$ and $e \in {\rm Id}(A)$.
\end{proposition}

\begin{proof}
Let $T(A, V)$ be CUSC. It is easy to see that $A$ remains CUSC too.

For the converse, suppose A is CUSC.\\

\medskip

{\it Claim}. If $(e, x)^2 = (e, x) \in {\rm T}(A, V)$, then $e^2 = e$ and $x = 0$.\\

\medskip

In fact, $(e, x)^2 = (e, x)$ gives $e^2 = e$ and $ex + xe = x$. So, by hypothesis, we have $2ex = x$. Then, multiplying by $e$ gives $2ex = ex$ and hence $ex = 0$, so that $x = 0$, proving the claim.\\

Let us write $$(e, 0) + (u, x) = (f, 0) + (v, x^\prime),$$ where $(u, x), (v, x^\prime) \in U (T(A, V))$ and $(e, 0), (f,0) \in {\rm Id} (T(A, V))$ with $(e, 0)(u, x) = (u, x)(e, 0)$ and $(f, 0)(v, x^\prime) = (v, x^\prime)(f, 0)$. Therefore, we write $e + u = f + v$, where $u, v \in U (A)$ and $e, f \in {\rm Id}(A)$ with $eu = ue$ and $fv = vf$. Thus, by hypothesis, $e = f$ forcing that $(e, 0) = (f, 0)$ and hence ${\rm T}(A, V)$ is CUSC, as claimed.
\end{proof}

Let $R$ be a ring and let $M$ be an $(R,R)$-bi-module which is a general ring (possibly with no unity) in which the equalities $(mn)r=m(nr)=(mr)n$ and $(rm)n=r(mn)$ hold for all $m,n\in M$ and $r\in R$. Then, the ideal-extension ${\rm I}(R,M)$ of $R$ by $M$ is defined to be the {\it additive abelian group} ${\rm I}(R,M)=R\oplus M$ with multiplication $(r,m)(s,n)=(rs,rn+ms+mn)$. Note that, if $R^{\prime}$ is a ring and $R^{\prime}=R\oplus K$, where $R$ is a subring and $K$ is an ideal of $R^{\prime}$, then $R^{\prime}\cong {\rm I}(R,K)$ holds.

\begin{proposition}\label{prop 2.18}
If ${\rm I}(R,M)$ is UUSC, then $R$ is UUSC.
\end{proposition}

\begin{proposition}\label{prop 2.19}
An ideal-extension ${\rm I}(R,M)$ is UUSC if the following conditions are satisfied:
\begin{enumerate}
\item[(a)]
$R$ is UUSC.
\item[(b)]
If $e\in {\rm Id}(R)$, then $em=me$ for all $m\in M$.
\item[(c)]
If $m\in M$, then $m+n+mn=0$
for some $n\in M$.
\end{enumerate}
\end{proposition}

\begin{proof}
In view of points (b) and (c), every unit of ${\rm I}(R,M)$ is of the form $(u,m)$, where $u\in U(R)$ and $m\in M$ as well as every idempotent of ${\rm I}(R,M)$ is of the form $(e,0)$, where $e\in {\rm Id}(R)$ (see, for instance, \cite[Proposition 7]{3}). Now, let $(u, m) = (e,0) +(v, m^{\prime})$, where $(u,m), (v,m^{\prime})$ are units of ${\rm I}(R,M)$, and $(e,0)$ is an idempotent of ${\rm I}(R,M)$. Moreover, $(e,0)( v,m^{\prime})=(v,m^{\prime})(e,0)$, so $u=e+v$ such that $ev=ve$. Therefore, in virtue of (a), we have $e=0$, and thus $(e,0)=(0,0)$, as pursued.
\end{proof}

\begin{remark}
It is easy to see that Propositions \ref{prop 2.18} and \ref{prop 2.19} hold also for CUSC rings.
\end{remark}

\section{Main Theorems}

A ring $R$ is called {\it semi-potent} if every one-sided ideal {\it not} contained in $J(R)$ contains a non-zero idempotent. Moreover, a semi-potent ring $R$ is called {\it potent} if idempotents lift modulo $J(R)$. Hereafter, the center of a ring $R$ is denoted by ${\rm Z}(R)$.

\begin{theorem}\label{theo3.1}
Let $R$ be a semi-potent ring. Then, the following are equivalent:

\begin{enumerate}
\item
$\dfrac{R}{J(R)}$ is UUSC.
\item
$\dfrac{R}{J(R)}$ is Boolean.
\item
$U(R)=1+J(R)$.
\item
$U(R) \subseteq u c n_0(R)=\left\{e+j | e^2=e \in {\rm Z}(R), j \in J(R)\right\}$.
\item
For each $a \in U(R)$, there exists a unique $e \in {\rm Id}(R)$ such that $a-e \in J(R)$.
\item
For each $a \in U(R)$, there exists $e \in {\rm Id}(R)$ such that $a-e\in J(R)$.
\end{enumerate}
\end{theorem}

\begin{proof}
(i) $\Rightarrow$ (ii): Since $R$ is semi-potent, $\bar{R} = \frac{R}{J(R)}$ is semi-potent. We now show that $\bar{R}$ is a reduced ring. Assume $a^2 = 0$ for some $ 0 \not = a \in \bar{R}$. Then, by \cite[Theorem 2.1]{10}, there exists $0 \not = e^2 = e\in \bar{R}$ such that $e \bar{R} e \cong {\rm M}_2 (S)$ for a non-trivial ring $S$. As $\bar{R}$ is UUSC, so $e \bar{R} e$ is UUSC by Proposition \ref{prop 2.4}. It is, however, {\it not} 2-good, so $1$ is not 2-good in $e \bar{R} e$ whence $1$ is {\it not} 2-good in ${\rm M}_2 (S)$ too, and this is a contradiction, because in any $2 \times 2$ matrix ring it is always true that

$$
\begin{pmatrix}
1  & 0\\
0 & 1
\end{pmatrix}=\begin{pmatrix}
1 & 1\\
1 & 0
\end{pmatrix}+
\begin{pmatrix}
0 & -1\\
-1 & 1
\end{pmatrix}.
$$

\noindent Hence, $\bar{R}$ is reduced and thus abelian. We next show that $\bar{R}$ is a Boolean ring. Assume on the contrary that $a^2 \not = a$ for some $a \in \bar{R}$. As $\bar{R}$ is semi-potent with $J(\bar{R}) = \{0\}$, $(a - a^2)\bar{R}$ contains a non-zero idempotent, say $e$. Write $e(a - a^2)=b$ with $b \in \bar{R}$. Then, it must be that $$e = e(a - a^2)b = ea\cdot e(1 - a)b = e(1 - a)\cdot eab$$. As $e \bar{R} e$ is reduced, both $ea$ and $e(1-a)$ are units of $e \bar{R} e$. Since $ea + e(1 - a) = e$, it manifestly follows from Proposition \ref{prop2.1} that $e \bar{R} \bar{e}$ is {\it not} UUSC. This, however, contradicts Proposition \ref{prop 2.4}. Hence, $\bar{R}$ is Boolean, as wanted.

(ii) $\Rightarrow$ (iii) $\Rightarrow$ (iv) $\Rightarrow$ (v) $\Rightarrow$ (vi): These implications follow immediately from \cite[Theorem 3.1]{1}.

(vi) $\Rightarrow$ (i): Let $\bar{u} = \bar{e} + \bar{v}$, where $\bar{e} \in {\rm Id}(\bar{R})$ and $\bar{u}, \bar{v} \in U (\bar{R})$ with $\bar{e} \bar{v} = \bar{v} \bar{e}$. As units lift modulo $J(R)$, we can assume $u, v \in U (R)$. By (vi), $u = f + j$ and $v = g + j^\prime$, where $f, g \in {\rm Id}(R)$ and $j, j^\prime \in J(R)$. Then, $f = u -j \in U (R)$, so $f = 1$. Similarly, $g = 1$. Therefore, $\bar{1} = \bar{e} + \bar{1}$, so that $\bar{e} = \bar{0}$ and hence $\bar{R}$ is UUSC, as desired.
\end{proof}

A ring $R$ is said to be left quasi-duo (resp., right quasi-duo) if every maximal left ideal (resp., maximal right ideal) of $R$ is an ideal.

\medskip

We now extract a series of assertions as follows.

\begin{corollary}
Every potent UUSC ring is left and right quasi-duo.
\end{corollary}

\begin{proof}
Let $M$ be a maximal left ideal in a ring $R$. Since $R$ is potent and UUSC, a consultation with Theorem \ref{theo3.1} assures that the quotient-ring $\bar{R} = \frac{R}{J(R)}$ is Boolean, so we have $\frac{\bar{R}}{\bar{M}} \cong \mathbb{Z}_2$. But then $$\frac{R}{M} \cong \frac{\bar{R}}{\bar{M}} \cong \mathbb{Z}_2$$ has two elements, so that $\frac{R}{M} = \{M, 1 + M \}$, and thus we conclude $R = M \cup (1 + M)$.

Now, let $x \in M$ and $r \in R$; we must show that $xr \in M$. This is, however, clear if $r \in M$; for otherwise $r = 1 + y, y \in M$. Then, $xr = x + xy \in M$, as required.
\end{proof}
	
\begin{corollary}\label{cor3.2}
Let $R$ be a regular ring. Then, $R$ is UUSC if, and only if, $R$ is Boolean.
\end{corollary}

\begin{proof}
We know that if $R$ is regular, then $J(R)=\{0\}$ and $R$ is semi-potent, as required.
\end{proof}

\begin{proposition}\label{prop3.3}
A ring $R$ is USC if, and only if, $R$ is simultaneously clean and CUSC.
\end{proposition}

\begin{proof}
It is self-evident.
\end{proof}

\begin{theorem}\label{theo3.9}
Let $R$ be a ring which is simultaneously CUSC (resp., UUSC) and semi-potent. Then, $2 \in J(R)$.
\end{theorem}

\begin{proof}
Let us assume in a way of contradiction that $2 \notin J(R)$. Since $R$ is semi-potent, there exists $0 \neq e^2=e \in 2 R$. Write $e=2a$ with $a \in R$. Therefore, $ea=ae$ and so $(1-3e)(1-3ae)=1$. Similarly, $(1-3ae)(1-3e)=1$ and thus $1-3e\in U(R)$. We, however, know that $1-2e\in U(R)$, whence
$$1-2 e=0+(1-2 e)=e+(1-3 e).$$ We, therefore, have two strongly clean decompositions for the element $1-2e$ while $R$ is CUSC (resp., UUSC). Hence, $e=0$, which is the desired contradiction.
\end{proof}

The following statement is pivotal.

\begin{theorem}\label{theo3.10}
Let $R$ be a CUSC (resp., UUSC) and let $\bar{R}=\dfrac{R}{J(R)}$. Then, the following hold:

\begin{enumerate}
\item
$1$ is not 2-good in $R$.
\item
For any $0 \neq e^2=e \in R$ and any $u_1, u_2 \in U(e R e)$, the inequality $u_1+u_2 \neq e$ is true.
\item
For any $n>1$, there does not exist $0 \neq e^2=e \in R$ such that $eRe \cong {\rm M}_n(S)$ for some ring $S$.
\item
$\bar{1}$ is not 2-good in $\bar{R}$.
\item
If $R$ is potent, then for any $\bar{0} \neq \bar{e}^2=\bar{e} \in \bar{R}$ and any $\bar{u}_1, \bar{u}_2 \in U(\bar{e} \bar{R} \bar{e})$, the inequality $\bar{u}_1+\bar{u}_2 \neq \bar{e}$ is true.
\item
If $R$ is potent, then, for any $n>1$, there does not exist $\bar{0} \neq \bar{e}^2=\bar{e} \in \bar{R}$ such that $\bar{e} \bar{R} \bar{e} \cong {\rm M}_n(S)$ for some ring $S$.
\end{enumerate}
\end{theorem}

\begin{proof}
\begin{enumerate}
\item
Write $1 = u_1 + u_2$, where $u_1, u_2 \in U(R)$. Then, $u_1 = 1 + (-u_2) = 0 + u_1$. As $R$ is UUSC, we get $1 = 0$ that is an absurd. For the CUSC case the arguments are analogous, because every unit is obviously a clean element.
\item
We know that $eRe$ is a subring of $R$ and $1_{eRe}=e$. Therefore, the result follows from (i).
\item
Since $n>1$, it is well known that ${\rm M}_n(S)$ contains a corner ring isomorphic to the $2 \times 2$ matrix ring. So, $eRe$ contains a corner ring isomorphic to the $2 \times 2$ matrix ring. Thus, without loss of generality, we can assume that $n=2$. Furthermore, in ${\rm M}_2(S)$, we have
$$\begin{pmatrix}
1 & 0 \\ 0 & 1\end{pmatrix}=\begin{pmatrix}
1 & 1 \\ 1 & 0\end{pmatrix}+\begin{pmatrix}
0 & -1 \\ -1 & 1\end{pmatrix}.$$ That is why, in the corner ring $eRe$, we have $u_1+u_2=e$. This is, however, against (ii), as wanted.
\item
And because units lift modulo $J(R)$, point (iv) follows immediately from (i).
\item
Given $\bar{e}, \bar{u}_1, \bar{u}_2$ as in (v), we can assume $e^2=e$ because idempotents lift modulo $J(R)$. Then, one has that $\bar{e} \bar{R} \bar{e} \cong \dfrac{e R e}{J(e R e)}$ exploiting \cite[Theorem 21.10]{6} since $eRe$ is CUSC (resp., UUSC) using Proposition \ref{prop 2.4}.
\item
Follows from (v).
\end{enumerate}
\end{proof}

We now arrive at our chief necessary and sufficient condition that is the transversal between USC and CUSC rings.

\begin{theorem}\label{theo3.4}
A ring $R$ is USC if, and only if, $R$ is CUSC and potent.
\end{theorem}

\begin{proof}
Suppose that $R$ is CUSC and potent. So, $R$ is semi-potent, and thus $\bar{R}=\dfrac{R}{J(R)}$ is semi-potent. We show that $\bar{R}$ is reduced. To that goal, assume that $\bar{t}^2=\overline{0}$, where $\overline{0} \neq \bar{t} \in \bar{R}$. Then, utilizing \cite[Lemma 14]{2}, there exists $\overline{0} \neq \bar{e}^{2}=\bar{e} \in \bar{R} $ such that $\bar{e} \bar{R} \bar{e} \cong {\rm M}_2(S)$ for some ring $S$. However, this contradicts Theorem \ref{theo3.10}. Hence, $\bar{R}$ is reduced, so it has to be abelian.

Now, we prove that $\bar{R}$ is USC. To that aim, suppose $$\bar{a}=\bar{e} + \bar{u}=\bar{f}+\bar{v}$$ are two strongly clean decompositions for $\bar{a}$ in $\bar{R}$. Consequently, $\bar{g}:=\bar{e}-\bar{f}$ is a central idempotent of $\bar{R}$, because $\bar{R}$ is abelian and because $\overline{2}=\overline{0}$ taking into account Theorem \ref{theo3.9}. Thus, $\bar{g}=\bar{v}+(-\bar{u})$, so that $\bar{g}=\bar{g} \bar{v} \bar{g}+\bar{g}(-\bar{u}) \bar{g}$, where both $\bar{g} \bar{v} \bar{g}$ and $\bar{g}(-\bar{u}) \bar{g}$ are units of $\bar{g} \bar{R} \bar{g}$. Bearing in mind Theorem \ref{theo3.10}, it must be that $\bar{g}=\overline{0}$, whence $\bar{e}=\bar{f}$. Hence, $\bar{R}$ is USC, as promised, and, moreover, it is also abelian. Therefore, $\bar{R}$ is UC and thus it is clean. As all idempotents lift modulo $J(R)$, we conclude that $R$ is clean referring to \cite[Proposition 6]{5}. Then, $R$ is CUSC and clean, so $R$ is USC by virtue of Proposition \ref{prop3.3}. The other implication is elementary.\\
\end{proof}

As three challenging consequences, we yield:

\begin{corollary}\label{cor3.5}
Let $R$ be a CUSC ring. Then, the following are equivalent:

\begin{enumerate}
\item
$R$ is clean.
\item
$R$ is exchange.
\item
$R$ is potent.
\item
$R$ is USC.
\item
$R$ is strongly clean.
\end{enumerate}
\end{corollary}

\begin{proof}
It is straightforward, so we omit the full details leaving them to the interested reader for an inspection.
\end{proof}

We shall say that a ring $R$ is {\it semi-boolean}, provided the factor-ring $\frac{R}{J(R)}$ is Boolean and idempotents lif modulo $J(R)$ or, equivalently, provided every element of $R$ is a sum of an idempotent and an element from $J(R)$.\\

\begin{corollary}\label{cor3.6}
A ring $R$ is semi-boolean if, and only if, $R$ is potent and $\dfrac{R}{J(R)}$ is UUSC.
\end{corollary}

\begin{proof}
Supposing that $R$ is semi-boolean, we get $R$ is clean, and thus it is potent and hence semi-potent. Besides, $\dfrac{R}{J(R)}$ is Boolean, so it is UUSC by Theorem \ref{theo3.1}.

Conversely, assuming $R$ is potent and $\dfrac{R}{J(R)}$ is UUSC, we get $R$ is semi-potent. Thus, $\dfrac{R}{J(R)}$ is Boolean by Theorem \ref{theo3.1} and, moreover, idempotents lift modulo $J(R)$, so that $R$ is semi-boolean.
\end{proof}

\begin{corollary}\label{cor3.8}
If $R=ucn_{0}(R)=\{ e+j| e^{2}=e\in {\rm Z}(R),j\in J(R)\}$, then $R$ is USC. The converse is not true.
\end{corollary}

\begin{proof}
If $R =ucn_{0}(R)$, then $R$ is uniquely clean by \cite[Corollary 3.4]{1}. Thus, $R$ is USC.

Next, put $R= {\rm T}_{2} (\mathbb{Z}_{2})$. So, $R$ is USC, but it is manifestly not UC. Also, $R\neq ucn_{0}(R)$, because
$$\begin{pmatrix}
1 & 1\\
0 & 0
\end{pmatrix}=\begin{pmatrix}
1 & 0\\
0 & 0
\end{pmatrix}+\begin{pmatrix}
0 & 1\\
0 & 0
\end{pmatrix},$$
where $\begin{pmatrix}
1 & 0\\
0 & 0
\end{pmatrix}$ is obviously not a central element.
\end{proof}

We are now prepared to establish the following crucial result, which gives a satisfactory necessary and sufficient condition for the triangular matrix ring to be CUSC, and also expands the corresponding one from \cite{2}.

\begin{theorem}\label{theo3.11}
Let $R$ be a commutative semi-potent ring. Then, the following statements are equivalent:
\begin{enumerate}
\item
$R$ is CUSC.
\item
$R$ is CUC.
\item
${\rm T}_n(R)$ is CUSC for all $n \geq 1$.
\item
${\rm T}_n(R)$ is CUSC for some $n \geq 1$.
\end{enumerate}
\end{theorem}

\begin{proof}
(i) $\Leftrightarrow$ (ii): This implication follows from Proposition \ref{prop2.1}. \\
(iii) $\Rightarrow$ (iv): This implication is obvious.\\
(iv) $\Rightarrow$ (i): Set $e=\mathrm{diag}(1,0, \ldots , 0) \in {\rm T}_n(R)$. Then, one inspects that $R \cong e {\rm T}_n(R) e$. Therefore, $R$ is CUSC by Proposition \ref{prop 2.4}.\\
(i) $\Rightarrow$ (iii): Clearly, the result is true for $n=1$. Assume now that the result holds for $n\geq 2$. Let $A=\begin{pmatrix} a_{11} & x\\
0 & A_{1}
\end{pmatrix}\in {\rm T}_{n+1}(R)$ is a clean element, where $a_{11}\in R$, $x\in M_{1\times n}(R)$ and $A_{1} \in {\rm T}_{n}(R)$. Then, both $A_{1}$ and $a_{11} $ are clean elements. By (i) and by hypothesis, $A_{1}$ and $a_{11}$ have unique strongly clean expressions in ${\rm T}_{n}(R)$ and in $R$, respectively, like this:
$A_{1}=E_{1}+U_{1}$, $a_{11}=e_{11}+u_{11}$.

Note that, if $$A=\begin{pmatrix}
e_{22} & y\\
0 & E_{2}
\end{pmatrix}+\begin{pmatrix}
u_{22} & z\\
0 & U_{2}
\end{pmatrix}$$ is a strongly clean expression in ${\rm T}_{n+1}(R)$, then $A_1=E_2+U_2$ and $a_{11}=e_{22}+u_{22}$ are strongly clean expressions in ${\rm T}_n(R)$ and in $R$, respectively (notice that $A$ is clean, so $A_1$ and $a_{11}$ are both clean). Hence, $E_1=E_2$, $U_1=U_2$, $e_{11}=e_{22}$ and $u_{11}=u_{22}$. Now, it suffices to prove that there exists a unique $x_1 \in {\rm M}_{1\times n}(R)$ such that $$A=\begin{pmatrix}
a_{11} & x \\
0 & A_1
\end{pmatrix}=
\begin{pmatrix}
e_{11} & x_1 \\
0 & E_1
\end{pmatrix}+
\begin{pmatrix}
u_{11} & x-x_1 \\
0 & U_1
\end{pmatrix}$$ is a strongly clean decomposition in ${\rm T}_{n+1}(R)$. Set $E:=\begin{pmatrix}
e_{11} & x_1 \\
0 & E_1
\end{pmatrix}$ and $U:=\begin{pmatrix}
u_{11} & x-x_1 \\
0 & U_1
\end{pmatrix}$. Apparently, $U$ is invertible. It is also easy to see that
\begin{align}
E^2=E & \Leftrightarrow e_{11} x_1+x_1 E_1=x_1 \nonumber\\
& \Leftrightarrow (e_{11} I+E_1) x_1=x_1 \label{eq1sta}\\
E U=U E & \Leftrightarrow e_{11}(x-x_1)+x_1 U_1=u_{11} x_1+(x-x_1) E_1 \nonumber\\
& \Leftrightarrow (U_1-u_{11} I-2 e_{11} I) x_1+(E_1+e_{11} I) x_1 \nonumber\\
& =(E_1-e_{11} I) x \label{eq2star}
\end{align}
Combining \eqref{eq1sta} with \eqref{eq2star}, we deduce
\begin{align*}
& (U_1-u_{11} I-2 e_{11} I) x_1+x_1=(E_1-e_{11} I) x ,\text{so that} \\
& [U_1+(1-2 e_{11}-u_{11}) I] x_1=(E_1-e_{11} I) x .
\end{align*}
Since $R$ is commutative semi-potent and CUSC, one verifies that $R$ is semi-potent CUC and $\dfrac{R}{J(R)}$ is Boolean (see, e.g., \cite[Theorem 3.2]{1}). Thus, $2 \in J(R)$ and $1-u^{\prime} \in J(R)$ for all $u^{\prime} \in U(R)$, because $U(R)=1+J(R)$. In particular, $1-u_{11} \in J(R)$ and so $1-2 e_{11}-u_{11} \in J(R)$. Therefore, $U_1+(1-2 e_{11}-u_{11}) I$ is invertible and hence we can write $$x_1=[U_1+(1-2 e_{11}-u_{11}) I]^{-1}(E_1-e_{11} I) x.$$ Then, $x_1$ has to be unique.

Next, we show that $x_1$ satisfies \eqref{eq1sta} and \eqref{eq2star}. As $E_1 U_1=U_1 E_1$, the elements $E_1+e_{11} I, U_1+(1-2 e_{11}-u_1) I$ and $[U_1+(1-2 e_{11}-u_{11}) I]^{-1}$ all commute. Consequently,
\begin{align*}
& (E_1+e_{11} I) x_1=[U_1+(1-2 e_{11}-u_{11}) I]^{-1}(E_1+e_{11} I) \\
& (E_1-e_{11} I) x=[U_1+(1-2 e_{11}-u_{11}) I]^{-1}(E_1-e_{11} I) x \\
& =x_1 .
\end{align*}
So, $x_1$ satisfies \eqref{eq1sta}. Similarly, $x_1$ satisfies \eqref{eq2star}, as pursued.
\end{proof}

A question which naturally arises is what happens in the case of UUSC rings?

\medskip

We are now concerned with group rings.\\

We denote by $RG$ the group ring of $G$ over $R$. The augmentation mapping $\varepsilon : RG \rightarrow R$ is given by $\varepsilon (\sum a_gg) = \sum a_{g}$ and its kernel, denoted by $\triangle (RG)$, is an ideal generated by $\{1-g| g \in G \}$. A group $G$ is called a {\it $p$-group} if every element of $G$ is power of $p$, where $p$ is a prime. Also, a group $G$ is called {\it locally finite} if every finitely generated subgroup is finite.

\begin{lemma}\label{lem4.1}
Suppose that every idempotent of $RG$ is contained in $R$. Then, $RG$ is CUSC (resp., UUSC) if, and only if, $R$ is CUSC (resp., UUSC).
\end{lemma}

\begin{proof}
In view of Proposition \ref{prop2.1}, it suffices to show the UUSC case. The necessity is clear by Proposition \ref{prop 2.4}. For the sufficiency, let $u=e+v$ where $u,v\in U(RG)$ and $e\in Id(RG)=Id(R)$ with $ev=ve$. Then $\varepsilon (u)=\varepsilon (e+v)=\varepsilon (e)+\varepsilon (v)= e+\varepsilon (v)$, where $\varepsilon (u), \varepsilon (v)\in U(R)$ and $e\in Id(R)$ with $e\varepsilon (v)=\varepsilon (v)e$. Since $R$ is UUSC, it follows that $e=0$. Then $RG$ is UUSC.
\end{proof}

As an immediate consequence, we extract:

\begin{corollary}\label{cor4.2}
The integral group ring $\mathbb{Z} G$ of an arbitrary group $G$ is CUSC.
\end{corollary}

The following two assertions are worthwhile.

\begin{proposition}\label{prop4.4}
Let $G$ be a locally finite $2$-group, and let $R$ be UUSC and semi-potent. Then, $RG$ is UUSC.
\end{proposition}

\begin{proof}
Let $u=e+v$ with $ev=ve$, where $u,v \in U(RG)$ and $e \in {\rm Id}(RG)$. Consequently, $$\varepsilon(u)=\varepsilon(e)+\varepsilon(v),
\varepsilon(e) \varepsilon(v)=\varepsilon(e v)=\varepsilon(v e)=\varepsilon(v) \varepsilon(e),$$ where
$\varepsilon(u), \varepsilon(v) \in U(R)$ and $\varepsilon(e) \in {\rm Id}(R)$. So, $\varepsilon(e)=0$, and thus $e \in \operatorname{ker} \varepsilon=\Delta (RG)$. We also know by Theorem \ref{theo3.9} that if $R$ is UUSC and semi-potent, then $2 \in J(R)$. Consequently, using \cite[Lemma 2]{4}, we have $\Delta (RG) \subseteq J(RG)$, so that $e \in J(RG)$ which assures that $e=0$.
\end{proof}

We finish off our work with the following criterion.

\begin{theorem}\label{prop4.3}
Let $R$ be a potent ring and let $G$ be a locally finite group. Then, $RG$ is CUSC (resp., UUSC) if, and only if, $R$ is CUSC (resp., UUSC) and $G$ is a $2$-group.
\end{theorem}

\begin{proof}
Firstly, we prove the CUSC case. To that purpose, assume that $R$ is CUSC and $G$ is a $2$-group. Since $R$ is a potent ring, one observes that $R$ is USC in virtue of Theorem \ref{theo3.4}. Therefore, $RG$ is USC owing to \cite[Theorem 3.1]{7}. Secondly, suppose that $RG$ is CUSC, so $R$ is CUSC by virtue of Proposition \ref{prop2.4}. Assume, in a way of contradiction, that $G$ is not a $2$-group. Then, $G$ contains an element $g$ of prime order $p>2$. As $R\langle g\rangle$ is a subring of $RG$, so $R\langle g\rangle$ is UUSC by Proposition \ref{prop 2.4}. Since $R$ is potent and UUSC, the factor $\dfrac{R}{J(R)}$ is Boolean employing Theorem \ref{theo3.1}, so $p\in U(R)$. And since
$$
x^{p-1}+x^{p-2}+\cdots + x+1=(x-1)[x^{p-2}+2x^{p-3}+\cdots +(p-2)x+(p-1)] +p,
$$ we derive $(x-1)(x^{p-1}+x^{p-2}+\cdots +x+1)=R[x]$. Hence, we obtain
$$R\langle g\rangle \cong \dfrac{R[x]}{\langle x^{p}-1\rangle }\cong \dfrac{R[x]}{\langle x-1\rangle}\oplus \dfrac{R[x]}{\langle x^{p-1}+\cdots +x+1\rangle }.$$ Thus,
$S:= \dfrac{R[x]}{\langle x^{p-1}+\cdots +x+1\rangle}$ is UUSC by Proposition \ref{prop 2.5}. However, in $S$, we have that $$(x^{2}+x)[x^{p-3}+x^{p-5}+\cdots +x^{2}+1]= x^{p-1}+\cdots +x=-1\in U(S),$$ whence $x,x+1\in U(S)$. But, we also have $x+1=0+(x+1)=1+x$ that are two strongly clean decompositions for $1+x\in U(R)$, contradicting our initial assumption. Finally, $G$ is a $2$-group, as asserted.

Next, we show the UUSC case. Let $R$ be UUSC and $G$ a $2$-group. Then, $RG$ is UUSC exploiting Proposition \ref{prop4.4}. Conversely, letting $RG$ be UUSC, then $R$ is UUSC consulting with Proposition \ref{prop 2.4}. Additionally, $G$ has to be a $2$-group using the same arguments as in the proof of the previous case.
\end{proof}

\medskip
\medskip
\medskip
	
\noindent{\bf Funding:} The first-named author, P.V. Danchev, of this research paper was partially supported by the Junta de Andaluc\'ia under Grant FQM 264, and by the BIDEB 2221 of T\"UB\'ITAK.

\end{document}